\documentclass[a4paper,reqno]{amsart}
\usepackage{microtype}
\usepackage[T1]{fontenc}
\usepackage[utf8]{inputenc}
\usepackage{lmodern}
\usepackage[hidelinks=true,hypertexnames=false]{hyperref}
\usepackage{enumitem}
\usepackage{csquotes}
\usepackage{upgreek}
\usepackage{amsmath,amssymb}
\usepackage{amsthm,thmtools}
\usepackage{mathtools}
\usepackage{stmaryrd}
\usepackage[nameinlink]{cleveref}
\usepackage{xcolor}
\usepackage{tikz}
\usepackage{fullpage}
\usepackage{orcidlink}

\usepackage[backend=biber,style=alphabetic,maxbibnames=20,maxalphanames=4,giveninits=true]{biblatex}
\usepackage[textsize=tiny,shadow]{todonotes}

\DeclareSymbolFont{bbold}{U}{bbold}{m}{n}
\DeclareSymbolFontAlphabet{\mathbbold}{bbold}

\newcommand{\N}{\mathbb{N}}

\newcommand{\R}{\mathbb{R}}
\newcommand{\C}{\mathbb{C}}

\newcommand{\from}{\colon}

\newcommand{\calL}{\mathcal{L}}
\newcommand{\calF}{\mathcal{F}}

\DeclareMathOperator{\dom}{dom}
\DeclareMathOperator{\Real}{Re}

\DeclareMathOperator{\ran}{ran}
\DeclareMathOperator{\spt}{spt}
\DeclareMathOperator{\ls}{span}
\renewcommand{\Re}{\Real}

\renewcommand{\i}{\mathrm{i}}
\newcommand{\multm}{\mathrm{m}}
\newcommand{\euler}{\mathrm{e}}

\DeclareMathOperator{\divergence}{div}
\renewcommand{\div}{\divergence}
\DeclareMathOperator{\grad}{grad}
\DeclareMathOperator{\sbb}{s_b}

\makeatletter

\newcommand\MyPairedDelimiter{%
  \@ifstar{\My@Paired@Delimiter{{}}}
          {\My@Paired@Delimiter{}}%
}
\newcommand\My@Paired@Delimiter[4]{%
  \newcommand#2{%
    \@ifstar{\start@PD{#1}{\delimitershortfall=-1sp}{#3}{#4}}
            {\start@PD{#1}{}{#3}{#4}}%
  }%
}
\newcommand\start@PD[5]{%
  #1\mathopen{\mathpalette\put@delim@helper{\put@delim{#2}{#3}{.}{#5}}}%
  #5%
  \mathclose{\mathpalette\put@delim@helper{\put@delim{#2}{.}{#4}{#5}}}%
}
\newcommand\put@delim@helper[2]{%
  \hbox{$\m@th\nulldelimiterspace=0pt #2#1$}%
}
\newcommand\put@delim[5]{%
  \setbox\z@\hbox{$\m@th#5{#4}$}%
  \setbox\tw@\null
  \ht\tw@\ht\z@ \dp\tw@\dp\z@
  #1#5%
  \left#2\box\tw@\right#3%
}

\makeatother
\MyPairedDelimiter*{\abs}{\lvert}{\rvert}
\MyPairedDelimiter*{\norm}{\lVert}{\rVert}
\MyPairedDelimiter{\set}{\{}{\}}

\providecommand{\scpr}[2]{\left( #1 \,\middle|\, #2 \right)}

\renewcommand{\sp}{\scpr}

\theoremstyle{plain} 
\newtheorem{theorem}{Theorem}[section]

\newtheorem{lemma}[theorem]{Lemma}

\theoremstyle{definition}
\newtheorem{example}[theorem]{Example}

\newtheorem{hypothesis}[theorem]{Hypothesis}
\newtheorem*{definition}{Definition}
\newtheorem{remark}[theorem]{Remark}

\usepackage{enumitem}

\title{Spatial Approximation for Evolutionary Equations}
\author[A.~Buchinger]{Andreas Buchinger\,\orcidlink{0009-0004-4203-5874}}
\address[A.B.]{Technische Universität Hamburg \\
  Institut für Mathematik \\
  Am Schwarzenberg-Campus 3 \\
  D-21073 Hamburg \\
  Germany}
\email{andreas.buchinger@tuhh.de}
\author[C.~Seifert]{Christian Seifert\,\orcidlink{0000-0001-9182-8687}}
\address[C.S.]{Technische Universität Hamburg \\
  Institut für Mathematik \\
  Am Schwarzenberg-Campus 3 \\
  D-21073 Hamburg \\
  Germany}
\email{christian.seifert@tuhh.de}
\author[S.~Trostorff]{Sascha Trostorff\,\orcidlink{0000-0002-5915-5782}}
\address[S.T.]{Christian-Albrechts-Universität zu Kiel \\
Mathematisches Seminar \\
Heinrich-Hecht-Platz 6\\
D-24098 Kiel\\
Germany}
\email{trostorff@math.uni-kiel.de}
\author[M.~Waurick]{Marcus Waurick\,\orcidlink{0000-0003-4498-3574}}
\address[M.W.]{TU Bergakademie Freiberg \\
  Institute of Applied Analysis \\
  Akademiestrasse 6 \\
  D-09599 Freiberg \\
  Germany}
\email{marcus.waurick@math.tu-freiberg.de}
\date{\today}

\begin{document}

\maketitle
\begin{abstract}
We consider evolutionary equations as introduced by R.\ Picard in 2009 and develop a general theory for approximation which can be seen as a theoretical foundation for numerical analysis for evolutionary equations.
To demonstrate the approximation result, we apply it to a spatial discretisation of the heat equation using spectral methods.


\end{abstract}

\section{Introduction}
\label{sec:Introduction}

Evolutionary equations describe a Hilbert space-approach to model systems of (time-dependent) partial differential equations (PDEs). They were introduced in the 2000s (\cite{Picard2009,PiMc11}) and provide a well-posedness theorem (see \Cref{thm:PicardWP}) that can be applied uniformly without the necessity to adapt it for each new model. For a detailed introduction into the theory of evolutionary equations and the corresponding established literature, see \cite{PicardMcGheeTrostorffWaurick2020,SeTrWa22}. The theory is shown to cover an ever-growing class of PDEs including classical ones, like the heat equation that we will use as our model example in this paper, but also advanced and partially novel models, e.g., poro- and fractional elasticity, (thermo-)piezoelectric coupling models, or (infinite-dimensional) differential-algebraic equations. We refer to \cite{McPic10,PicTroWau15,Picard17,MPTW16,BuDo24,TroWau2019} for this short and incomplete list. Apart from application to different models, the theory of evolutionary equations is also subject to a constant extension to different fields in the realm of (partial) differential equations. We shall mention abstract boundary data spaces (cf., e.g., \cite{PicTrosWau2016,PicSeiTroWau16}),
maximal regularity (cf.~\cite{PTW17,TW21}), homogenisation (most recently \cite{BEW24,BSW24}), delay-differential equations (cf., e.g., \cite{KPSTW14,FW24,AW24}), and control theory (cf.~\cite{BS24}) as some examples. 
Additionally, evolutionary equations have already been treated numerically in the case of particular problems of mixed type with an emphasis on the time discretisation (see~\cite{FTW2019}), and in the setting of particular homogenisation problems (see~\cite{FrWa18,BFSW24}). A discontinuous Galerkin approach was used in time;
a finite element method in space. Up to now, these numerical approaches are lacking a general background theory in the framework of evolutionary equations that readily yields (a rate of) convergence of the approximations. Thus, we will establish a first version of such a theory for the spatial component in this paper. The treatment of the temporal component and, ultimately, the combination of both components and application to
the (homogenisation) problems and the corresponding approximation schemes in~\cite{FTW2019,FrWa18,BFSW24} are topics of ongoing research. 

Let us now briefly introduce the notion of (autonomous) evolutionary equations and the aim of this paper.
Consider a system of partial differential equations with autonomous coefficients. Let $H$ be the Hilbert space of spatial functions, i.e., in terms of this paper, $L_2(\Omega)^k$ for some $k\in\N$ and a domain $\Omega\subseteq\R^d$, where $d\in\N$.
The solution theory is set in $L_{2,\nu}(\R,H)$. Here, $\nu$ denotes the parameter of an exponential weight in time $\R$ that, for $\nu\neq 0$, renders the weak (time) derivative $\partial_{t,\nu}$ invertible.
Let $M$ be a material law, i.e., a holomorphic function from a certain subset of $\C$ to the bounded operators on $H$ describing the coupling of the system, and let $A$ be a skew-selfadjoint operator on $H$ describing the appearing spatial derivatives. Then, the corresponding evolutionary equation reads
\[(\partial_{t,\nu} M(\partial_{t,\nu}) + A)u = f\]
for some source term $f\in L_{2,\nu}(\R,H)$, where $M(\partial_{t,\nu})$ is defined by means of a holomorphic functional calculus. Imposing usual conditions on $M$, we obtain well-posedness and the solution is given by $u = S_\nu f$, where $S_\nu\coloneqq (\overline{\partial_{t,\nu} M(\partial_{t,\nu}) + A})^{-1}$ is a bounded operator on
$L_{2,\nu}(\R,H)$.
For $n\in\N$, let $H_n$ be Hilbert spaces (in the easiest case, increasing $n$-dimensional subspaces of $H$), $M_n$ a material law, and $A_n$ a skew-selfadjoint operator, all with respect to $H_n$. Consider the evolutionary equations
\[(\partial_{t,\nu} M_n(\partial_{t,\nu}) + A_n)u_n = f_n,n\in\N,\]
for some right-hand sides $f_n\in L_{2,\nu}(\R,H_n)$. For the solutions given by $u_n = S_{n,\nu} f_n$, where $S_{n,\nu}\coloneqq (\overline{\partial_{t,\nu} M_n(\partial_{t,\nu}) + A_n})^{-1}$, we have the following ultimate goal: find suitable conditions that imply convergence of $u_n$ (embedded in $L_{2,\nu}(\R,H)$, cf.~\cite{KP2018,V81,IOS89,KS2003,PS2520,PZ2023}) to $u$ in different topologies such that existing common numerical approximations that have already been applied to evolutionary equations can be embedded in this framework. We note in passing that the first proof of well-posedness of evolutionary equations in \cite{Picard2009} can be viewed as a particular space-time version of the approximation result established here. Moreover, the main abstract homogenisation result in \cite{Wau14} is based on a particular instance of the results established here. Furthermore, we expect applications to problems with underlying spatial domains changing their shape along the lines of \cite{KP2018,PZ2023}, which also form as a source of our inspiration.

Finally, we outline how this paper is organised. In \cref{sec:EvolutionaryEquations}, we introduce the relevant theory concerning evolutionary equations in detail. In \cref{sec:SpatialApproximation}, we obtain the
main convergence results. In \cref{sec:ExHeatEq1d}, we use spectral methods to treat the (autonomous) heat equation with different boundary conditions
in one and in higher dimensions as a first example and a proof of principle.
Finite element methods---in particular for Maxwell's equations (compare also~\cite{ER2020,DES2021,ER2021} and the references therein)---are yet to be treated.

\section{Evolutionary Equations}
\label{sec:EvolutionaryEquations}

In this section, we will recall the needed concepts and theorems of the theory of evolutionary equations (see~\cite{SeTrWa22} for proofs and details).

Our model problem will be the heat equation on some open and bounded $\Omega\subseteq\R^d$ with Dirichlet
boundary conditions and in time $\R$. We formally write
\begin{equation}
\partial_t \theta -\div\grad_0 \theta =Q\text{.}\label{eq:HeatEq1Line}
\end{equation}
In order to write this equation in the evolutionary way and to discuss its well-posedness, we need a few definitions. For $\nu\in\R$ and a Hilbert space $H$, we write $L_{2,\nu}(\R,H)$ for the space of all (equivalence
 classes of) Bochner-measurable $f\colon\R\to H$ with $\euler^{-\nu\cdot}\norm{f(\cdot)}_H\in L_{2}(\R)$,
and we define
\begin{equation*}
\exp(-\nu\multm)\colon
\begin{cases}
\hfill L_{2,\nu} (\R,H) &\to\quad L_{2} (\R,H)\\
\hfill f&\mapsto\quad \euler^{-\nu\cdot}f(\cdot).
\end{cases}
\end{equation*}
In composition with the common unitary Fourier transformation $\calF$ on $L_{2} (\R,H)$, this gives rise to the unitary Fourier--Laplace transformation
\[\calL^H_\nu\coloneqq\calF\exp(-\nu\multm)\colon L_{2,\nu} (\R,H) \to L_2 (\R,H)\text{.}\]

Let
\[\partial_{t,\nu}\coloneqq (\calL^H_\nu)^\ast (\i\multm +\nu)\calL^H_\nu\]
 denote the usual weak (time) derivative on $ L_{2,\nu} (\R,H)$, where $\multm\colon\dom(\multm)\subseteq L_2 (\R,H)\to L_2 (\R,H)$ with $\multm f \coloneqq  (x\mapsto xf(x))$ and maximal domain.
 
 Next, for $O\subseteq\C$ open, we call a holomorphic function
 \[M\colon O\to L(H)\coloneqq\{B\colon H\to H\mid B\text{ linear and bounded}\}\]
 such that $O$ contains $\C_{\Re>\mu}\coloneqq\{z\in\C\mid\Re z>\mu\}$ and $\sup_{z\in\C_{\Re>\mu}}\norm{M(z)}<\infty$ for some $\mu\in\R$, a material law, and we write $\sbb (M)$ for the infimum over all such $\mu$. 
 Assuming $\nu >\sbb(M)$ and using the bounded linear operator
  \begin{equation*}
M(\i\multm+\nu)\colon
\begin{cases}
\hfill L_{2} (\R,H) &\to\quad L_{2} (\R,H)\\
\hfill f&\mapsto\quad M(\i\cdot+\nu)f(\cdot),
\end{cases}
\end{equation*}
 we call
\[ M(\partial_{t,\nu})\coloneqq (\calL^H_\nu)^\ast M(\i\multm+\nu)\calL^H_\nu\in L(L_{2,\nu} (\R,H))\]
 the corresponding material law operator.

Finally, let $A\colon\dom(A)\subseteq H\to H$ be skew-selfadjoint. Then, we readily obtain the following lifted
version to $L_{2,\nu} (\R,H)$ that we will still call $A$; the properties of which are well understood:
\begin{lemma}\label{lemma:LiftedSpatialOperatorA}
Let $H$ be a Hilbert space, $\nu\in\R$, and  $A\colon\dom(A)\subseteq H\to H$ skew-selfadjoint.
Regard $\dom (A)$ as a Hilbert space endowed with the graph inner product induced by $A$.
Then,
 \begin{equation*}
 A\colon
 \begin{cases}
\hfill L_{2,\nu}(\R,\dom(A))\subseteq L_{2,\nu} (\R,H) &\to\quad L_{2,\nu} (\R,H)\\
\hfill f&\mapsto\quad (t\mapsto A(f(t)))
\end{cases}
 \end{equation*}
 is skew-selfadjoint.
\end{lemma}
Altogether, these concepts and definitions allow us to
present the autonomous evolutionary equation
\begin{equation}\label{eq:StdAutEvoEqu}
(\partial_{t,\nu} M(\partial_{t,\nu}) + A)u = f
\end{equation}
for suitable $u,f\in L_{2,\nu} (\R,H)$, and to formulate the following essential well-posedness theorem.
\begin{theorem}[Picard's Well-Posedness Theorem]\label{thm:PicardWP}
Let $H$ be a Hilbert space, $A\colon\dom(A)\subseteq H\to H$ skew-selfadjoint, $M$ a material law, 
$\nu>\sbb(M)$, and $c>0$ such that
\[\Re \sp{h}{zM(z)h}_H\geq c\norm{h}^2_H\]
for all $h\in H$ and $z\in\C$ with $\Re z\geq \nu$.

Then, $\partial_{t,\mu} M(\partial_{t,\mu})+A$ is closable with bounded inverse
$S_\mu\coloneqq (\overline{\partial_{t,\mu} M(\partial_{t,\mu}) + A})^{-1}\in L(L_{2,\mu} (\R,H))$
and $\norm{S_\mu}\leq 1/c$ for all $\mu\geq\nu$. In particular, $S_\mu$ is the material law operator corresponding to the material law
$z\mapsto (zM(z)+A)^{-1}$ defined for $z\in\C$ with $\Re z\geq\nu$.

Moreover, $\spt f\subseteq [a,\infty)$ for $a\in\R$ and $f\in L_{2,\mu} (\R,H)$ implies $\spt S_\mu f\subseteq [a,\infty)$. For $\mu_1,\mu_2\geq\nu$, we have $S_{\mu_1}f=S_{\mu_2}f$ for all
$f\in  L_{2,\mu_1} (\R,H)\cap L_{2,\mu_2} (\R,H)$. Finally, $f\in\dom(\partial_{t,\mu})$ yields $S_\mu f\in \dom(\partial_{t,\mu})\cap\dom (A)$, i.e., due to $ M(\partial_{t,\mu})\partial_{t,\mu}\subseteq \partial_{t,\mu} M(\partial_{t,\mu})$, the solution $u\coloneqq S_\mu f$ solves the evolutionary equation in the literal sense of~\labelcref{eq:StdAutEvoEqu}. 
\end{theorem}
\begin{proof}
See, e.g.,~\cite[Theorem~6.2.1]{SeTrWa22}.
\end{proof}
Going back to~\labelcref{eq:HeatEq1Line}, we set the spatial Hilbert space $H\coloneqq L_2(\Omega)\times L_2(\Omega)^d$, the coupling of the system via the material law
$z\mapsto M(z) \coloneqq \begin{psmallmatrix}1 & 0\\0 & 0\end{psmallmatrix} + z^{-1}\begin{psmallmatrix}0 & 0\\0 & 1\end{psmallmatrix}$ for
$z\in\C$ with $\Re z> 0$, and the spatial differential operators via the skew-selfadjoint unbounded operator $A\coloneqq\begin{psmallmatrix} 0 & \div\\\grad_0 & 0\end{psmallmatrix}$. Here, $\div$ is the weak divergence on $L_2(\Omega)^d$ with its maximal domain, and $\grad_0$ is the closure of the gradient on the test functions $C_c^\infty(\Omega)\subseteq L_2(\Omega)$ with respect to the graph norm, i.e., the weak gradient on $L_2(\Omega)$
with domain $H^1_0(\Omega)$. By means of these definitions, we consider~\labelcref{eq:HeatEq1Line} as the evolutionary equation
 \begin{equation}\label{eq:HeatEqEvolEq}
 \Bigl(\partial_{t,\nu} \begin{pmatrix}1 & 0\\0 & 0\end{pmatrix} + \begin{pmatrix}0 & 0\\0 & 1\end{pmatrix} + \begin{pmatrix} 0 & \div\\\grad_0 & 0\end{pmatrix}\Bigr) \begin{pmatrix}\theta\\q \end{pmatrix} = \begin{pmatrix} Q\\0 \end{pmatrix}
 \end{equation}
for $\nu>0$ and a source $Q\in L_{2,\nu}(\R, L_2(\Omega))$. We can easily verify that
\Cref{thm:PicardWP} can be applied to~\labelcref{eq:HeatEqEvolEq}, see also \Cref{sec:ExHeatEq1d}.
\section{Spatial Approximation}
\label{sec:SpatialApproximation}

In this section, we will consider spatial approximations of evolutionary equations. We will rely on the notation introduced in Section \ref{sec:EvolutionaryEquations}. Before we dive into the particulars of our setting, we provide one particular setting we have in mind.
\begin{remark}\label{rem:OrthProJCanInjIntro}
Consider an (increasing) sequence of finite-dimensional subspaces $(H_n)_n$ of a Hilbert space $H$ such that $\bigcup_n H_n$ is dense in $H$. Then, $n\in\N$, let $P_n\from H\to H_n$ be the corresponding orthogonal projections and $J_n\from H_n\to H$ the corresponding canonical embeddings. Then,
\[\norm{P_n} = \norm{J_n} = 1,\]
and $P_nJ_n = I_{H_n}$ for all $n\in\N$. The density assumption yields $J_nP_n\to I$ in the strong operator topology. 
Note that this setting is remindful of a Galerkin approximation, where the sequence $(H_n)_{n\in\N}$ is indeed usually constructed to be increasing.
\end{remark}
We will now provide a more general framework and convergence results that will cover first numerical approximations schemes in the context of evolutionary equations.

For the remaining section we will make the following assumptions.
\begin{hypothesis}
\label{hyp:Pn}
For $n\in\N$ let $H$ and $H_n$ be Hilbert spaces,
$M$ and $M_n$  material laws corresponding to $H$ and $H_n$ respectively, and
$A\colon\dom(A)\subseteq H\to H$ and $A_n\colon\dom(A_n)\subseteq H_n\to H_n$ skew-selfadjoint operators.

Furthermore, for $n\in\N$, regard $\dom(A)$ and $\dom(A_n)$ as Hilbert spaces endowed with the respective graph inner product and let $D\subseteq\dom (A)$ be a dense subspace of $H$. Assume that $D$, endowed with some inner product, is a Hilbert space that continuously embeds into $\dom(A)$.

For $n\in\N$, let $P_n\in L(D,\dom(A_n))$ and $J_n\in L(\dom(A_n),\dom(A))$ such that
 \begin{equation}\label{eq:ConvJnPnStronglyToI}
\sup_{n\in\N}\norm{P_n}<\infty\quad\text{ and }\quad\sup_{n\in\N}\norm{J_n}<\infty\quad\text{ and }\quad J_n P_n f\to f
 \end{equation}
  for all $f\in D$ as $n\to \infty$, where the convergence holds with respect to the graph norm of $A$. Furthermore, for each $n\in\N$, let there be an extension $\tilde{J}_n\in L(H_n,H)$ of
$J_n$ such that
\[\sup_{n\in\N}\norm{\tilde{J}_n}<\infty\text{.}\]

Assume there exists $\nu_0\in \R$ such that
\[\sbb(M), \sup_{n\in\N}\sbb(M_n) \leq \nu_0.\]

Moreover, let $c>0$ such that for all $n\in\N$ and $z\in\C$ with $\Re z>\nu_0$ we have
\begin{equation}\label{eq:CoerciveMnAndMmeansStability}
\forall h\in H:\Re \sp{h}{zM(z)h}_H\geq c\norm{h}^2_H\quad\text{ and }\quad \forall h\in H_n:\Re \sp{h}{zM_n(z)h}_{H_n} \geq c\norm{h}^2_{H_n}.
\end{equation}

Finally, we will additionally assume that for each $z\in\C$ with $\Re z >\nu_0$ there exists a $d_z>0$ with
\begin{equation}
\label{eq:Mn_uniformly_bounded}
    \sup_{n\in\N} \norm{zM_n(z)} \leq d_z,
\end{equation}
i.e., that $(z\mapsto zM_n(z))_{n\in\N}$ is a pointwise bounded sequence.
\end{hypothesis}
\begin{definition}
     Let $T_n\in L(H_n)$, $T\in L(H)$, $0\in \rho(T+A)$ and, recalling that $D\subseteq \dom(A)$, assume $(T+A)^{-1}(D)\subseteq D$. We say that 
    \begin{itemize}
     \item \emph{$(T_n+A_n)_{n\in\N}$ converges strongly to $T+A$ on $D$} if
     \[\norm{P_n (T+A)f- (T_n+A_n) P_n f}_{H_n}\to 0\]
     for all $f\in (T+A)^{-1}(D)$ as $n\to\infty$.
     \item \emph{$(T_n+A_n)_{n\in\N}$ converges compact-weakly to $T+A$ on $D$} if
     \[\sp{g_n}{P_n (T+A)f-(T_n+A_n)P_n f}_{H_n}\to 0\]
     for all $f\in (T+A)^{-1}(D)$ as $n\to\infty$ where $g_n\in \dom (A_n)$ for $n\in\N$ and $\sup_{n\in\N}\norm{g_n}_{\dom (A_n)}<\infty$.
    \end{itemize}
\end{definition}
\begin{remark}
Since we will ultimately ask for convergence of the material laws with respect to these modified versions of the weak and strong operator topology,
we need to assume~\labelcref{eq:Mn_uniformly_bounded} instead of deducing this property by virtue of the uniform boundedness principle.
\end{remark}
\begin{remark}\label{rem:JTAPstrongConvFollows}
In the setting of the previous definition, we can make the following reformulation.

Let $(T_n+A_n)_{n\in\N}$ converge strongly to $T+A$ on $D$.
Then, for $f\in (T+A)^{-1}(D)$,
    \begin{multline*}
    \norm{(T+A)f-\tilde{J}_n(T_n+A_n)P_n f}_{H}\\
    \leq \sup_{n\in\N}\norm{\tilde{J}_n}\norm{P_n (T+A)f- (T_n+A_n) P_n f}_{H_n}+\norm{(T+A)f-J_nP_n(T+A)f}_H\to 0
    \end{multline*}
    as $n\to \infty$.
    
    Note that we are not able to provide an analogous statement for compact-weak convergence since we do not assume anything concerning the mappings
    $\tilde{J}_n^\ast\restriction_{\dom(A)}$ at this point. In contrast, see \Cref{rem:generalization_lem:isem23_14.1.2_weak}. 
\end{remark}
\begin{remark}\label{rem:UniformCompEmbAn}
The name of the previous definition of compact-weak convergence comes from the fact that it resembles a certain \enquote{uniform compact embedding} of the domains $\dom(A_n), n\in\N$. The standard case in applications will
be $H_n\subseteq H$ and $\dom(A_n)\subseteq \dom(A)$, where $\dom(A)$ as a Hilbert space embeds compactly into $H$. Then, we can define \emph{weak convergence of $(T_n+A_n)_{n\in\N}$
to $T+A$ on $D$} as
 \begin{equation}\label{eq:WOTConvOnDSpCase}
 \sp{g}{P_n (T+A)f-(T_n+A_n)P_n f}_{H}\to 0
 \end{equation}
     for all $f\in (T+A)^{-1}(D)$ and $g\in H$ as $n\to\infty$. By the principle of uniform boundedness and since any bounded sequence in $\dom(A)$ has a (strongly) convergent subsequence in $H$, this implies compact-weak convergence (compare also \Cref{lem:weak-strong_weak}). Alternatively, we could restrict~\labelcref{eq:WOTConvOnDSpCase} to $g\in\dom(A)$ and instead ask for $\sup_{n\in\N}\norm{P_n (T+A)f-(T_n+A_n)P_n f}_{H}<\infty$ for each $f\in (T+A)^{-1}(D)$.

 For an in-depth discussion on general weak convergence for sequences in varying Hilbert spaces, see, e.g.,~\cite[Section~2.2]{KS2003}.
\end{remark}
With that, we obtain the following first central lemma, which is related to~\cite[Lemma~14.1.2]{SeTrWa22}. Before, however, providing the lemma, a standard invertibility result is presented. The proof of which can be found in \cite[Chapter 6]{SeTrWa22}.
\begin{lemma}\label{lem:wp_easy} Let $T\in L(H)$, $A$ a skew-selfadjoint operator in $H$. If there is $c>0$ such that
\[
 \Re \sp{h}{T h}_{H}\geq c\norm{h}^2_{H}\quad (h\in H)
\]
then $0\in \rho(T+A)$.
\end{lemma}

\begin{lemma}
\label{lem:isem23_14.1.2}
	For $n\in\N$, let $T_n\in L(H_n)$, $T\in L(H)$, and assume that there exist $c,d>0$ with
	\[\norm{T_n}\leq d\quad\text{ and }\quad\Re \sp{h}{T_n h}_{H_n}\geq c\norm{h}^2_{H_n}\]
	 for all $n\in\N$ and $h\in H_n$ as well as
	 \begin{equation}\label{eq:CoerLimitTStrongImplStrong}
	 \Re \sp{h}{T h}_{H}\geq c\norm{h}^2_{H}
	 \end{equation}
	 for all $h\in H$. Then, $0\in\bigcap_{n\in\N}\rho(T_n+A_n)\cap \rho(T+A)$. 
	 
	 Next, let $D\subseteq \dom(A)$ such that $(T+A)^{-1}(D)\subseteq D$ and assume that $(T_n+A_n)_{n\in\N}$ converges strongly to $T+A$ on $D$.
	Then, 
	\[J_n(T_n+A_n)^{-1}P_n \to (T+A)^{-1} \]
	in the strong operator topology on $L(D,\dom(A))$ as $n\to\infty$.
\end{lemma}

\begin{proof}
	The invertibility statement follows from \Cref{lem:wp_easy}. By virtue of
	$A_n(T_n+A_n)^{-1}=I-T_n(T_n+A_n)^{-1}$,
	\begin{equation}\label{eq:UniformBoundTnAnRes}
	\sup_{n\in\N} \norm{(T_n+A_n)^{-1}}_{L(H_n,\dom(A_n))}\leq \frac{\sqrt{1+c^2+d^2}}{c},
	\end{equation}
	and in particular, we obtain that $J_n(T_n+A_n)^{-1}\in L(H_n,\dom (A))$ is uniformly bounded.

 Now, let $f\in D$ and compute
	\begin{equation}\label{eq:RearrangeConvResWSImpl}
		 \begin{split}
		 J_n(T_n+A_n)&^{-1}P_n f - (T+A)^{-1}f\\
		& = \bigl(J_n(T_n+A_n)^{-1}P_n(T+A) - (J_nP_n+I-J_nP_n)\bigr)(T+A)^{-1}f\\
		& = \bigl(J_n(T_n+A_n)^{-1}P_n(T+A) - J_nP_n\bigr)(T+A)^{-1}f - (I-J_nP_n)(T+A)^{-1}f\\
		& = J_n(T_n+A_n)^{-1}\bigl(P_n(T+A) -(T_n+A_n)P_n\bigr)(T+A)^{-1}f - (I-J_nP_n)(T+A)^{-1}f.
	\end{split}
	\end{equation}
	Thus, the claimed convergence follows since $(J_n(T_n+A_n)^{-1})_{n\in\N}$ is uniformly bounded from $H_n$ to $\dom(A)$ and $J_nP_n(T+A)^{-1}f\to (T+A)^{-1}f$ in $\dom (A)$.
\end{proof}
\begin{remark}[Convergence Rates]
	In the setting of \Cref{lem:isem23_14.1.2}, let $g\from \N\to [0,\infty)$ such that 
	\[\norm{f-J_nP_nf}_{\dom(A)}, \norm{P_n(T+A)f-(T_n+A_n)P_nf}_{H_n}\leq g(n) \norm{f}_D\]
	 for all $f\in (T+A)^{-1}(D)$ and all $n\in\N$. Then, by~\labelcref{eq:RearrangeConvResWSImpl},
	\[\norm{J_n(T_n+A_n)^{-1}P_nf - (T+A)^{-1}f}_{\dom(A)} \leq  \Bigl(\sup_{n\in\N}\norm{J_n}\frac{\sqrt{1+c^2+d^2}}{c}+1\Bigr) g(n) \norm{(T+A)^{-1}f}_D\]
	for all $f\in D$ and all $n\in\N$.
	
	Furthermore, let $h\from \N\to [0,\infty)$ such that 
	\[\norm{f-J_nP_nf}_H, \norm{P_n(T+A)f-(T_n+A_n)P_nf}_{H_n}\leq h(n) \norm{f}_D\]
	 for all $f\in (T+A)^{-1}(D)$ and all $n\in\N$. Then, by~\labelcref{eq:RearrangeConvResWSImpl},
	\[\norm{J_n(T_n+A_n)^{-1}P_nf - (T+A)^{-1}f}_H \leq \Bigl(\sup_{n\in\N}\norm{\tilde{J}_n}\frac{1}{c}+1\Bigr) h(n) \norm{(T+A)^{-1}f}_D\]
	for all $f\in D$ and all $n\in\N$.
\end{remark}
\begin{remark}
Under the assumptions of \Cref{lem:isem23_14.1.2}, we obtain
\begin{multline*}
\norm{(T_n+A_n)^{-1}P_n f -P_n(T+A)^{-1}f}_{H_n}\\
\leq \frac{1}{c}\norm{P_n(T+A)(T+A)^{-1}f -(T_n+A_n)P_n (T+A)^{-1}f}_{H_n}\to 0
\end{multline*}
for $f\in D$ as $n\to \infty$. Together with~\labelcref{eq:ConvJnPnStronglyToI}, this shows that we are treating a convergence similar to the one
in~\cite[Section~2]{KP2018}; see also~\cite{V81}, \cite{IOS89}, \cite[Section~2.3]{KS2003}, \cite{PS2520}, \cite[Section~1.5]{PZ2023}, and the references therein.  
\end{remark}

\begin{remark}\label{rem:generalization_lem:isem23_14.1.2}
In the spirit of \Cref{rem:OrthProJCanInjIntro},
    if $D=\dom(A)$, the $H_n$ are closed subspaces of $H$, the $P_n$ are ($\dom(A)$-$\dom(A_n)$ stable; i.e., $\sup_n\|P_n\|_{L(\dom(A),\dom(A_n))}<\infty$) restrictions of the orthogonal projections $\tilde{P}_n\in L(H,H_n)$, the $J_n$ are restrictions of the canonical embeddings $\tilde{J}_n\in L(H_n,H)$, and $(T+A)^{-1}(\dom(A))\subseteq H$ is dense, then we do not need to assume~\labelcref{eq:CoerLimitTStrongImplStrong} in \Cref{lem:isem23_14.1.2}: 
    
    Indeed, let $h\in (T+A)^{-1}(\dom(A))\subseteq H$. Then
	\begin{multline}\label{eq:WOTimpliesCoercLimit}
	\Re \sp{h}{Th}_{H} = \lim_{n\to\infty} \Re \sp{h}{\tilde{J}_n(T_n+A_n)P_nh}_{H} = \lim_{n\to\infty} \Re\sp{P_n h}{T_n P_n h}_{H_n}\\
	 \geq c \lim_{n\to\infty} \norm{P_n h}_{H_n}^2 = c \lim_{n\to\infty} \norm{J_n P_n h}_{H}^2 = c\norm{h}_H^2\text{.}
	\end{multline}
	Hence, by a density argument, \labelcref{eq:CoerLimitTStrongImplStrong} follows.
\end{remark}
Before we prove the second central \Cref{lem:isem23_14.1.2_weak}, which is once again related to~\cite[Lemma~14.1.2]{SeTrWa22}, we state the following easy and well-known convergence result on which we will base our proof of \Cref{lem:isem23_14.1.2_weak}; compare \Cref{rem:UniformCompEmbAn}.
\begin{lemma}\label{lem:weak-strong_weak}
    Let $K_0,K_1$ be Hilbert spaces, $(B_n)_{n\in\N}$ a sequence in $L(K_0,K_1)$, $B\in L(K_0,K_1)$, $B_n\to B$ in the weak operator topology, $(u_n)_{n\in\N}$ a sequence in  $K_0$, $u\in K_0$, and $u_n\to u$ in $K_0$.
    
    Then, $B_nu_n\to Bu$ weakly in $K_1$.
\end{lemma}

\begin{proof}
    Since $B_n\to B$ weakly, $(B_n)_{n\in\N}$ is bounded by the uniform boundedness principle. Let $v\in K_1$. Then,
    \[
        \sp{v}{B_n u_n}_{K_1} - \sp{v}{Bu}_{K_1}  = \sp{B_n^* v}{u_n-u}_{K_0} + \sp{v}{(B_n-B)u}_{K_1} \to 0. \qedhere
    \]
\end{proof}

\begin{lemma}
\label{lem:isem23_14.1.2_weak}
	For $n\in\N$, let $T_n\in L(H_n)$, $T\in L(H)$, and assume that there exist $c,d>0$ with
	\[\norm{T_n}\leq d\quad\text{ and }\Re \sp{h}{T_n h}_{H_n}\geq c\norm{h}^2_{H_n}\]
	 for all $n\in\N$ and $h\in H_n$ as well as
	 \begin{equation}\label{eq:CoerLimitTWeakImplStrong}
	 \Re \sp{h}{T h}_{H}\geq c\norm{h}^2_{H}
	 \end{equation}
	 for all $h\in H$. Assume $D\subseteq \dom(A)$ with $(T+A)^{-1}D\subseteq D$.
	Further, assume that $\dom(A)$ as a Hilbert space embeds compactly into $H$ and that $(T_n+A_n)_{n\in\N}$ converges compact-weakly to $T+A$ on $D$.
	
	Then,
	$0\in\bigcap_{n\in\N}\rho(T_n+A_n)\cap \rho(T+A)$ and
	\[J_n(T_n+A_n)^{-1}P_n  \to (T+A)^{-1}\]
	in the strong operator topology on $L(D,H)$ as $n\to\infty$.
\end{lemma}

\begin{proof}
	First, we obtain $0\in\bigcap_{n\in\N}\rho(T_n+A_n)\cap \rho(T+A)$ applying \Cref{lem:wp_easy}.
	
	Let $f\in D$ and $u_n\coloneqq J_n(T_n+A_n)^{-1}P_nf$ for $n\in\N$. Then, $(u_n)_{n\in\N}$ is bounded in $\dom(A)$. Hence, there exists a subsequence $(u_{n_k})_{k\in\N}$ which converges weakly to some $u\in \dom(A)$, and by compactness of the embedding, $u_{n_k}\to u$ in $H$ as $k\to\infty$.
	As in~\labelcref{eq:RearrangeConvResWSImpl} we obtain
	\begin{multline*}
	J_{n_k}(T_{n_k}+A_{n_k})^{-1}P_{n_k} f - (T+A)^{-1}f\\ = J_{n_k}(T_{n_k}+A_{n_k})^{-1}\bigl(P_{n_k}(T+A) -(T_{n_k}+A_{n_k})P_{n_k}\bigr)(T+A)^{-1}f - (I-J_{n_k}P_{n_k})(T+A)^{-1}f
	\end{multline*}
	for $k\in\N$.
	We have $J_{n_k}P_{n_k}(T+A)^{-1} f\to (T+A)^{-1}f$ in $\dom(A)$, so also in $H$. Furthermore, for $v\in H$, we have (by skew-selfadjointness of the $A_{n_k}$)
	\begin{multline*}
        \sp{v}{J_{n_k}(T_{n_k}+A_{n_k})^{-1}\bigl(P_{n_k}(T+A) -(T_{n_k}+A_{n_k})P_{n_k}\bigr)(T+A)^{-1}f }_H  \\
        = \sp{(T_{n_k}^\ast-A_{n_k})^{-1}\tilde{J}_{n_k}^\ast v}{\bigl(P_{n_k}(T+A) -(T_{n_k}+A_{n_k})P_{n_k}\bigr)(T+A)^{-1}f}_{H_{n_k}}.
	\end{multline*}
    With \Cref{lem:wp_easy} we infer
    \[\sup_{k\in\N}\norm{(T^\ast_{n_k}-A_{n_k})^{-1}}_{L(H_{n_k})}\leq\frac{1}{c}\]
    and since $A_{n_k}(T^\ast_{n_k}-A_{n_k})^{-1}=T^\ast_{n_k}(T^\ast_{n_k}-A_{n_k})^{-1}-1$, we infer that
    \[
     \sup_{k\in \N} \norm{(T^\ast_{n_k}-A_{n_k})^{-1}}_{L(H_{n_k},\dom(A_{n_k}))}<\infty.
    \]
    Thus,
    \[J_{n_k}(T_{n_k}+A_{n_k})^{-1}P_{n_k} f \to (T+A)^{-1}f\]
    weakly in $H$ as $n\to\infty$,
    which yields
    \[u = (T+A)^{-1} f.\]
     Now, since we have identified the limit $u$ (which is independent of the subsequence $(n_k)_{k\in\N}$), a subsequence-subsequence argument yields $u_n\to u$, or put differently,
    \[J_n(T_n+A_n)^{-1}P_n \to (T+A)^{-1}\]
    in the strong operator topology on $L(D,H)$ as $n\to\infty$.
\end{proof}

\begin{remark}\label{rem:generalization_lem:isem23_14.1.2_weak}
    In the spirit of \Cref{rem:generalization_lem:isem23_14.1.2} we discuss \Cref{lem:isem23_14.1.2_weak} in the following special case: assume that we are in the setting of  \Cref{rem:generalization_lem:isem23_14.1.2}, i.e., $D=\dom(A)$, the $H_n$ are closed subspaces of $H$, the $P_n$ are ($\dom(A)$-$\dom(A_n)$ stable) restrictions of the orthogonal projections $\tilde{P}_n\in L(H,H_n)$, the $J_n$ are restrictions of the canonical embeddings $\tilde{J}_n\in L(H_n,H)$, and $(T+A)^{-1}(\dom(A))\subseteq H$ is dense.
    
 We first observe that, similarly to \Cref{rem:JTAPstrongConvFollows}, for $f\in (T+A)^{-1}(\dom(A))\subseteq H$ and $g\in\dom (A)$
 \begin{multline*}
  \sp{g}{(T+A)f-\tilde{J}_n(T_n+A_n)P_n f}_{H}\\
  =\sp{P_n g}{P_n(T+A)f-(T_n+A_n)P_n f}_{H_n} + \sp{g}{(T+A)f-J_nP_n(T+A)f}_{H}\to 0
 \end{multline*}
 as $n\to\infty$.
 Thus, we can use~\labelcref{eq:WOTimpliesCoercLimit} to drop the assumption~\labelcref{eq:CoerLimitTWeakImplStrong}.
\end{remark}
We will now lift \Cref{lem:isem23_14.1.2} and \Cref{lem:isem23_14.1.2_weak} to the corresponding weighted $L_{2,\nu}$-spaces by virtue of the Fourier--Laplace transformation.

Applying the same principle as in \Cref{lemma:LiftedSpatialOperatorA}, for $n\in\N$ and $\nu\in\R$, we obtain the lifted operator
$P_n\in L(L_{2,\nu}(\R,D),L_{2,\nu}(\R,H_n))$, defined as
 \begin{equation*}
 P_n\colon
 \begin{cases}
\hfill L_{2,\nu}(\R,D) &\to\quad L_{2,\nu} (\R,H_n)\\
\hfill f&\mapsto\quad (t\mapsto P_n(f(t)))
\end{cases}
 \end{equation*}
 and similarly $\tilde{J}_n\in L(L_{2,\nu}(\R,H_n),L_{2,\nu}(\R,H))$ for $n\in\N$. Clearly, $\sup_{n\in\N}\norm{P_n}<\infty$ and $\sup_{n\in\N}\norm{\tilde{J}_n}<\infty$ still hold for these lifted operators.
\begin{lemma}
	\label{prop:isem23_13.1.2}
	Let $N$ and $N_n$ be material laws taking values in $L(H)$ and $L(H_n)$, respectively, such that
\[\sbb(N), \sup_{n\in\N}\sbb(N_n) \leq \nu_0\]
and
\begin{equation*}
    \sup_{n\in\N} \sup_{z\in \C_{\Re >\nu_0}} \norm{N_n(z)} \leq d,
\end{equation*}
i.e., that $(N_n)_{n\in\N}$ is a uniformly bounded sequence.
Moreover, for $z\in \C$ with $\Re z>\nu_0$, assume that
	$\tilde{J}_n N_n(z)P_n\to N(z)$ in the strong (weak) operator topology on $L(D,H)$ as $n\to\infty$. Let $\nu>\nu_0$. Then,
	\[\tilde{J}_nN_n(\partial_{t,\nu})P_n\to N(\partial_{t,\nu})\]
	 in the strong (weak) operator topology on $L(L_{2,\nu}(\R,D),L_{2,\nu}(\R,H))$ as $n\to\infty$.
\end{lemma}

\begin{proof}
	For $\varphi\in C_c^\infty (\R,D)\subseteq L_{2,\nu}(\R,D)$ and $n\in\N$ we obtain $P_n\varphi\in C_c^\infty (\R,H_n)\subseteq L_{2,\nu}(\R,H_n)$ and
	\begin{multline*}
	\calL_\nu^{H_n}P_n\varphi=t\mapsto \frac{1}{\sqrt{2\uppi}}\int_{\R}\exp(-(it+\nu)s)P_n\varphi(s)\mathrm{d}s\\
	=t\mapsto P_n\Bigl(\frac{1}{\sqrt{2\uppi}}\int_{\R}\exp(-(it+\nu)s)\varphi(s)\mathrm{d}s\Bigr)=P_n \calL_\nu^{D}\varphi\text{.}
	\end{multline*}
	Hence, by density, this implies $\calL_\nu^{H_n}P_n=P_n \calL_\nu^{D}$ for $n\in\N$. Analogously, we infer $\tilde{J}_n (\calL_\nu^{H_n})^\ast=(\calL_\nu^{H})^\ast \tilde{J}_n$ for $n\in\N$. Note that we lifted $P_n$ and $\tilde{J}_n$ with respect to the weight parameter $\nu$ on the left-hand sides of both identities and with
	respect to the respective unweighted spaces on the right-hand sides. Thus,
	\[\tilde{J}_n N_n(\partial_{t,\nu})P_n= \tilde{J}_n (\calL^{H_n}_\nu)^\ast N_n(\i\multm+\nu)\calL^{H_n}_\nu P_n=(\calL_\nu^{H})^\ast \tilde{J}_n N_n(\i\multm+\nu) P_n \calL_\nu^{D}\]
	holds for $n\in\N$. As $\calL_\nu^{D}$ and $\calL_\nu^{H}$ are both unitary mappings it remains to show that
	$\tilde{J}_n N_n(\i\mathrm{m} + \nu) P_n\to N(\i \mathrm{m} + \nu)$ in the strong (weak) operator topology on $L(L_{2}(\R,D),L_{2}(\R,H))$ as $n\to\infty$, which
	immediately follows from the dominated convergence theorem.	
\end{proof}
With these tools at hand we can prove the following generalisation of \cite[Theorem~13.1.5]{SeTrWa22}.
\begin{theorem}\label{thm:convergence_spatial}
    Recall $D\subseteq \dom(A)$.
    For $z\in \C$ with $\Re z>\nu_0$, let $(zM(z)+A)^{-1}(D)\subseteq D$ and $zM_n(z)+A_n\to zM(z)+A$ strongly on $D$ as $n\to\infty$. Moreover, let $\nu>\nu_0$. Then,
     \[\tilde{J}_n\bigl(\overline{\partial_{t,\nu}M_n(\partial_{t,\nu}) +A_n}\bigr)^{-1}P_n\to \bigl(\overline{\partial_{t,\nu}M(\partial_{t,\nu}) +A}\bigr)^{-1}\]
      in the strong operator topology on $L(L_{2,\nu}(\R,D),L_{2,\nu}(\R,H))$ as $n\to\infty$.
\end{theorem}

\begin{proof}
	Let $z\in \C$ with $\Re z>\nu_0$. 
	Then, \Cref{lem:isem23_14.1.2} applied to $T_n\coloneqq zM_n(z)$ for $n\in\N$ and $T\coloneqq zM(z)$ yields
	\[\tilde{J}_n(zM_n(z)+A_n)^{-1}P_n \to (zM(z)+A)^{-1}\]
	in the strong operator topology on $L(D,H)$ as $n\to\infty$.
	Thus, \Cref{thm:PicardWP} and \Cref{prop:isem23_13.1.2} applied to $N_n(z):=(zM_n(z)+A_n)^{-1}$ and $N(z):=(zM(z)+A)^{-1}$ for $z\in \C$ with $\Re z>\nu_0$ yield the assertion.
\end{proof}

\begin{remark}
    The strong convergence $M_n(z)+A_n\to M(z)+A$ on $D$ for all $z\in \C$ with $\Re z>\nu_0$ can be interpreted as a consistency condition.
    Moreover, the conditions~\labelcref{eq:CoerciveMnAndMmeansStability} and~\labelcref{eq:Mn_uniformly_bounded} are stability-type conditions. 
    Thus, Theorem \ref{thm:convergence_spatial} yields that consistency plus stability imply convergence of the approximated solutions to the actual one. Hence, \Cref{thm:convergence_spatial} is a version of the well-known meta-theorem in approximation theory and numerical analysis.
\end{remark}
The following theorem is generalisation of \cite[Theorem~14.1.1]{SeTrWa22}.
\begin{theorem}\label{thm:convergence_spatial_weak}
    Recall $D\subseteq \dom(A)$.
   For $z\in \C$ with $\Re z>\nu_0$, let $ (zM(z)+A)^{-1}(D)\subseteq D$ and $zM_n(z)+A_n\to zM(z)+A$ compact-weakly on $D$ as $n\to\infty$. Moreover, assume that $\dom(A)$ as a Hilbert space embeds compactly into $H$, and let $\nu>\nu_0$. Then,
     \[\tilde{J}_n\bigl(\overline{\partial_{t,\nu}M_n(\partial_{t,\nu}) +A_n}\bigr)^{-1}P_n\to \bigl(\overline{\partial_{t,\nu}M(\partial_{t,\nu}) +A}\bigr)^{-1}\]
      in the strong operator topology on $L(L_{2,\nu}(\R,D),L_{2,\nu}(\R,H))$ as $n\to\infty$.
\end{theorem}

\begin{proof}
	Let $z\in \C$ with $\Re z>\nu_0$. Then, Lemma \ref{lem:isem23_14.1.2_weak} applied to $T_n\coloneqq zM_n(z)$ for $n\in\N$ and $T\coloneqq zM(z)$ yields
	\[\tilde{J}_n(zM_n(z)+A_n)^{-1}P_n \to (zM(z)+A)^{-1}\]
	in the strong operator topology on $L(D,H)$ as $n\to\infty$.
	Thus, \Cref{thm:PicardWP} and \Cref{prop:isem23_13.1.2} applied to $N_n(z):=(zM_n(z)+A_n)^{-1}$ and $N(z):=(zM(z)+A)^{-1}$ for $z\in \C$ with $\Re z>\nu_0$ yield the assertion. 
\end{proof}

\section{Example: The Heat Equation}
\label{sec:ExHeatEq1d}
In this section, we will apply the results from \cref{sec:SpatialApproximation} to approximate the heat equation in different dimensions with different boundary conditions using spectral methods.
\begin{example}\label{ex:Spectr1DMixedBC}
 Let $\Omega\coloneqq (0,1)$ be the unit interval and $H\coloneqq L_2(\Omega)\times L_2(\Omega)$.
 We consider the heat equation with homogeneous mixed boundary conditions; that is
\[\Bigl(\partial_{t,\nu} \begin{pmatrix}1 & 0\\0 & 0\end{pmatrix} + \begin{pmatrix}0 & 0\\0 & 1\end{pmatrix} + \begin{pmatrix} 0 & \partial^{\{1\}}_{x}\\\partial^{\{0\}}_{x} & 0\end{pmatrix}\Bigr) \begin{pmatrix}\theta\\q\end{pmatrix} = \begin{pmatrix} Q\\0\end{pmatrix}\]
 for $\nu>0$ and a source $Q\in L_{2,\nu}(\R, H)$. Then, $\dom(A)= H^1_{\{0\}}(\Omega)\times H^1_{\{1\}}(\Omega)$, where, for $y\in\{0,1\}$, by virtue of the Sobolev embedding theorem,
 $H^1_{\{y\}}(\Omega)\coloneqq\{f\in H^1(\Omega) \mid f(y)=0\}$ and $\partial^{\{y\}}_{x}$ is the restriction of the usual weak derivative from $H^1(\Omega)$ to $H^1_{\{y\}}(\Omega)$. Integration by parts yields $(\partial^{\{0\}}_{x})^\ast=-\partial^{\{1\}}_{x}$, i.e., $A$ is skew-selfadjoint.
 Moreover, $z\mapsto M(z) \coloneqq \begin{psmallmatrix}1 & 0\\0 & 0\end{psmallmatrix} + z^{-1}\begin{psmallmatrix}0 & 0\\0 & 1\end{psmallmatrix}$ is a material law with $\sbb(M) = 0$, and for $\nu>0$ and $z\in \C$ with $\Re z\geq \nu$ we have
 \begin{equation}\label{eq:Coercivity1DHeatEqOneSidedBD}
 \Re \sp{(x,y)}{zM(z)(x,y)}_{H}\geq \min\{\nu,1\} \norm{(x,y)}_{H}^2
 \end{equation}
 for all $(x,y)\in L_2(\Omega)\times L_2(\Omega)=H$.
 
We readily obtain $\ker (A)=\{0\}$ and thus, by virtue of the Rellich--Kondrachov theorem, $\ran(A)=H$ with a compact resolvent $A^{-1}\in L(H)$. Hence, by the spectral theorem for self-adjoint compact operators applied to $iA^{-1}$, we obtain an orthonormal basis $(\varphi_j)_{j\in\N}$ of $H$ consisting of eigenvectors of $A^{-1}$ corresponding to nonzero eigenvalues $(\lambda_j)_{j\in\N}$ in $\i\R$. For $j\in\N$, we clearly have $\varphi_j\in\dom(A)$ with $A\varphi_j= \lambda_j^{-1} \varphi_j$.
Moreover, for $v\in\dom(A)$, we obtain
\begin{equation}\label{eq:SpecTheAInv}
Av=\sum_{j=1}^\infty \sp{\varphi_j}{Av}_{H} \varphi_j= -\sum_{j=1}^\infty\sp{A\varphi_j}{v}_{H}\varphi_j
=\sum_{j=1}^\infty \lambda_j^{-1} \sp{\varphi_j}{v}_{H} \varphi_j\text{.}
\end{equation}
For $n\in\N$, let $H_n\coloneqq \ls \{\varphi_j\mid j\leq n\}$ and define
\[
 \pi_n\colon
 \begin{cases}
\hfill H &\to\quad H_n\\
\hfill u&\mapsto\quad \sum_{j=1}^n \sp{\varphi_j}{u}_{H} \varphi_j
\end{cases}
\text{,}\]
i.e., the orthogonal projection with restricted codomain, as well as the canonical embedding $\iota_n\colon H_n\to H$; note that $\pi_n=\iota_n^*$.
Setting $A_n\coloneqq \pi_n A\iota_n$,
we infer $\dom(A_n)=H_n$ and that $A_n$ is bounded and skew-selfadjoint.
$M_n\coloneqq \pi_n M\iota_n$ is a material law with $\sbb(M_n)=0$, and, for $\nu>0$ and $z\in\C_{\Re\geq \nu}$, we have~\eqref{eq:Coercivity1DHeatEqOneSidedBD} for  all $(x,y)\in H_n$.  Furthermore, the stability condition is satisfied, i.e., $\sup_{n\in\N} \norm{zM_n(z)} \leq \max\{\abs{z},1\}$. In view of \Cref{hyp:Pn}, we set
$D\coloneqq\dom(A)$ and $P_n\coloneqq\pi_n\restriction_{\dom(A)}$ for $n\in\N$. Indeed, we then have $P_n\in L(\dom(A),\dom(A_n))$ with an operator norm independent of $n\in\N$ since, for $v\in\dom(A)$, we have 
\begin{equation}\label{eq:AAnStabilHProj}
\norm{A_n P_n v}_{H_n}^2= \sum_{j=1}^n \abs{\lambda_j}^{-2} \abs{\sp{\varphi_j}{v}_{H}}^2
\leq \sum_{j=1}^\infty \abs{\lambda_j}^{-2} \abs{\sp{\varphi_j}{v}_{H}}^2=\norm{Av}_{H}^2\text{.}
\end{equation}
With $\tilde{J}_n\coloneqq \iota_n\in L(H_n,H)$ for $n\in\N$, we also immediately have $J_n\in L(\dom(A_n),\dom(A))$ for the same mapping, where all operator norms are bounded by $1$. For $v\in\dom(A)$, we
obtain
\begin{align}
\begin{aligned}\label{eq:AConvHProjSpecMeth}
J_nP_n v&=\iota_n\pi_n v=\sum_{j=1}^n \sp{\varphi_j}{v}_{H}\varphi_j\to  \sum_{j=1}^\infty \sp{\varphi_j}{v}_{H}\varphi_j=v\quad\text{and}\\
AJ_nP_n v&=A\iota_n\pi_n v=\sum_{j=1}^n \lambda_j^{-1} \sp{\varphi_j}{v}_{H}\varphi_j\to \sum_{j=1}^\infty \lambda_j^{-1} \sp{\varphi_j}{v}_{H} \varphi_j=Av
\end{aligned}
\end{align}
in $H$ as $n\to\infty$. Finally, for $z\in\C$ with $\Re z > 0$, invertibility of $zM(z)+A$ yields
\[(zM(z)+A)(D)=
(zM(z)+A)(\dom(A))=H\supseteq\dom(A)=D\text{,}\]
and, with~\eqref{eq:AConvHProjSpecMeth}, we infer
 \begin{multline*}
 \norm{P_n (zM(z)+A)v- (zM_n(z)+A_n) P_n v}_{H_n}\\
 =\norm{\pi_nzM(z)v+\pi_n Av- \pi_n zM(z)\iota_n\pi_n v+\pi_n A \iota_n\pi_n v}_{H_n}\\
\leq \norm{zM(z)(v-\iota_n\pi_n v)}_{H}+\norm{A(v-\iota_n\pi_n v)}_{H}
 \to 0
 \end{multline*}
     for all $v\in (zM(z)+A)^{-1}(\dom(A))\subseteq\dom(A)$ as $n\to\infty$.
\end{example}

All in all, we can apply \Cref{thm:convergence_spatial} and conclude with the following observation.
\begin{theorem}
\label{thm:convergence_heatequation_1d}
Using the notions from \Cref{ex:Spectr1DMixedBC}, for $\nu >0$, we have
     \begin{multline*}
     \iota_n\Biggl(\overline{\partial_{t,\nu} \pi_n\begin{pmatrix}1 & 0\\0 & 0\end{pmatrix}\iota_n + \pi_n\begin{pmatrix}0 & 0\\0 & 1\end{pmatrix}\iota_n + \pi_n\begin{pmatrix} 0 & \partial^{\{1\}}_{x}\\\partial^{\{0\}}_{x} & 0\end{pmatrix}\iota_n}\Biggr)^{-1}\pi_n\restriction_{\dom(A)}\\
     \to \Biggl(\overline{\partial_{t,\nu} \begin{pmatrix}1 & 0\\0 & 0\end{pmatrix} + \begin{pmatrix}0 & 0\\0 & 1\end{pmatrix} + \begin{pmatrix} 0 & \partial^{\{1\}}_{x}\\\partial^{\{0\}}_{x} & 0\end{pmatrix}}\Biggr)^{-1}
     \end{multline*}
      in the strong operator topology on $L(L_{2,\nu}(\R,H^1_{\{0\}}(\Omega)\times H^1_{\{1\}}(\Omega)),L_{2,\nu}(\R,L_2(\Omega)\times L_2(\Omega)))$ as $n\to\infty$.
\end{theorem}
\begin{remark}
In order to discuss finite element methods in the setting of \cref{sec:SpatialApproximation}, we would need to work with the following spaces.

 For $k\in\N_0$ we write $\mathcal{P}_k$ for the vector space of polynomials from $\R$ to $\C$ of degree at most $k$.
 For $n\in\N$ let $h\coloneqq\frac{1}{n}$ and consider the partition $x_0<x_1<\ldots<x_n$  of $[0,1]$ with $x_j\coloneqq\frac{j}{n} = jh$ for all $j\in\{0,\ldots,n\}$, and let
 \begin{align*}
   S_n^1(\Omega) & \coloneqq\set{f\in L_2(\Omega):\; f\restriction_{(x_j,x_{j+1})}\in\mathcal{P}_1 \quad(j\in\{0,\ldots,n-1\})},\\
    S_n^{1,0}(\Omega) & \coloneqq\set{f\in L_2(\Omega):\; f\in C(\overline{\Omega}),\, f\restriction_{(x_j,x_{j+1})}\in\mathcal{P}_1 \quad(j\in\{0,\ldots,n-1\})},\\
    S_{n,0}^{1,0}(\Omega) & \coloneqq\set{f\in L_2(\Omega):\; f\in C(\overline{\Omega}),\, f\restriction_{(x_j,x_{j+1})} \in\mathcal{P}_1\quad(j\in\{0,\ldots,n-1\}),\, f(0) = f(1)=0}
 \end{align*} 
 be the vector spaces of (globally continuous) linear splines (with zero boundary conditions).
Then, $S_n^1(\Omega)$,  $S_{n}^{1,0}(\Omega)$, and $S_{n,0}^{1,0}(\Omega)$ are finite-dimensional subspaces of $L_2(\Omega)$ and hence Hilbert spaces. Note that we even have $S_n^{1,0}(\Omega)= H^1(\Omega)\cap S_n^1(\Omega)$ and $S_{n,0}^{1,0}(\Omega)= H^1_0(\Omega)\cap S_n^1(\Omega)$.

Unfortunately, there is no canonical way to define $A_n$ since a mere restriction of, e.g., $\partial_x$ (as used
in finite element methods) would read $\partial_{x,n}\colon S_n^{1,0}(\Omega)\subseteq S_n^1(\Omega)\to S_n^1(\Omega)$. This mapping clearly fails to be skew-selfadjoint or even densely defined. In fact, there seems to be no readily obtainable and sensible way to define
$A_n$ on these spaces; even when we are not asking for skew-selfadjointness.
\end{remark}
\begin{example}
If we just change the boundary conditions in \Cref{ex:Spectr1DMixedBC} to periodic ones, i.e., we consider
\[\Bigl(\partial_{t,\nu} \begin{pmatrix}1 & 0\\0 & 0\end{pmatrix} + \begin{pmatrix}0 & 0\\0 & 1\end{pmatrix} + \begin{pmatrix} 0 & \partial^{\#}_x\\\partial^{\#}_x & 0\end{pmatrix}\Bigr) \begin{pmatrix}\theta\\q\end{pmatrix} = \begin{pmatrix} Q\\0\end{pmatrix}\text{,}\]
where $\partial^{\#}_x$ is the restriction of the usual weak derivative from $H^1(\Omega)$ to $H^1_{\#}(\Omega)\coloneqq\{f\in H^1(\Omega) \mid f(0)=f(1)\}$, then we can obtain a result similar to \Cref{thm:convergence_heatequation_1d}. By virtue of the Rellich--Kondrachov theorem and the skew-selfadjointness of $A$, $\ran(A)=\ker(A)^\perp$ is closed with a compact resolvent $(A\restriction_{\ran(A)})^{-1}\in L(\ran(A))$. Thus, we can account for $H_0\coloneqq\ker(A)=\ls \{(1,0),(0,1)\}$ by adding a corresponding orthonormal basis $\varphi_0,\varphi_{-1}\in \ker(A)$ and introducing the orthogonal projection $\pi_0$, the canonical
embedding $\iota_0$, and the arising $P_0,J_0,\tilde{J}_0$, as well as, for each $n\in\N$, including $\varphi_0,\varphi_{-1}$ in the
construction of $\pi_n$, $\iota_n$, $H_n$, $P_n$, $J_n$, and $\tilde{J}_n$ in \Cref{ex:Spectr1DMixedBC}.
\end{example}
\begin{example}
If we consider our original model problem~\eqref{eq:HeatEqEvolEq}, i.e., the heat equation with Dirichlet boundary conditions on some open and bounded $\Omega\subseteq\R^d$,
 \begin{equation*}
 \Bigl(\partial_{t,\nu} \begin{pmatrix}1 & 0\\0 & 0\end{pmatrix} + \begin{pmatrix}0 & 0\\0 & 1\end{pmatrix} + \begin{pmatrix} 0 & \div\\\grad_0 & 0\end{pmatrix}\Bigr) \begin{pmatrix}\theta\\q\end{pmatrix} = \begin{pmatrix} Q\\0\end{pmatrix}\text{,}
 \end{equation*}
 then we can once again obtain a result similar to \Cref{thm:convergence_heatequation_1d}.
 By virtue of the Rellich--Kondrachov theorem $H^1_0(\Omega)$ compactly embeds into $L_2(\Omega)$.
 Therefore, we can also deduce that $\dom(\div)\cap\ker(\div)^\perp$ compactly embeds into $L_2(\Omega)^d$.
 This and the skew-selfadjointness of $A$ imply $\ran(A)=\ker(A)^\perp$ is closed with a compact resolvent $(A\restriction_{\ran(A)})^{-1}\in L(\ran(A))$. Thus, we can account for $\ker(A)$ by adding a corresponding
 orthonormal basis $\varphi_0,\varphi_{-1},\dots\in \ker(A)$, introducing the orthogonal projection $\pi_0$, the canonical
embedding $\iota_0$, and the arising $P_0,J_0,\tilde{J}_0$ corresponding to $H_0\coloneqq\ls\{\varphi_0\}$ as well as, for each $n\in\N$,
adding $\varphi_0,\dots,\varphi_{-n}$ in the
construction of $\pi_n$, $\iota_n$, $H_n$, $P_n$, $J_n$, and $\tilde{J}_n$ in \Cref{ex:Spectr1DMixedBC}.
\end{example}
\begin{remark}
In finite element methods, stability under the differential operators of the projections onto the elements is a
key property that needs to be satisfied. We shall mention the question when the $L_2$-projection is $H^1$-stable (see, e.g., \cite[Folgerungen II.7.8, II.7.9]{Braess2007}) for the gradient and the properties of the Raviart--Thomas projection (see, e.g., \cite{GarciaJensen1997}, \cite[Minimaleigenschaft III.5.3]{Braess2007}, and
\cite[Theorem 4.1 and Theorem 5.1]{AcostaDuran1999}) for the divergence as prominent examples.
The analogue in \Cref{hyp:Pn} is $P_n\in L(D,\dom(A_n))$ for $n\in\N$. In the examples presented in this section, this means stability of the $L_2$-projections with respect to $\dom(A)$; and that immediately follows since we project onto eigenspaces of $A$, see~\eqref{eq:SpecTheAInv},~\eqref{eq:AAnStabilHProj},
and~\eqref{eq:AConvHProjSpecMeth}.
\end{remark}

\printbibliography

\end{document}